\documentclass[11pt]{amsart}

\usepackage[a4paper,hmargin=3.5cm,vmargin=3.5cm]{geometry}
\usepackage{amsfonts,amssymb,amscd,amstext}
\usepackage{graphicx}
\usepackage[dvips]{epsfig}
\usepackage{quoting}
\usepackage{hyperref}

\usepackage{fancyhdr}
\input xy
\xyoption{all}

\usepackage{enumerate}
\usepackage{titlesec}
\usepackage{mathrsfs}

\usepackage{tikz-cd} 

\titleformat{\section}
{\filcenter\bfseries} {\thesection{.}}{0.2cm}{}
\titleformat{\subsection}[runin]
{\bfseries} {\thesubsection{.}}{0.15cm}{}[.]
\titleformat{\subsubsection}[runin]
{\em}{\thesubsubsection{.}}{0.15cm}{}[.]

\usepackage[up,bf]{caption}

\newtheorem{theorem}{Theorem}[section]
\newtheorem{proposition}[theorem]{Proposition}
\newtheorem{lemma}[theorem]{Lemma}
\newtheorem{claim}[theorem]{Claim}
\newtheorem{corollary}[theorem]{Corollary}

\theoremstyle{definition}
\newtheorem{definition}[theorem]{Definition}
\newtheorem{remark}[theorem]{Remark}

\newtheorem{problem}[theorem]{Problem}

\numberwithin{equation}{section}
\numberwithin{figure}{section}

\newcommand\Ascr{\mathscr{A}}

\newcommand\Cscr{\mathscr{C}}

\newcommand\Iscr{\mathscr{I}}

\newcommand\Oscr{\mathscr{O}}

\newcommand\Gscr{\mathscr{G}}

\newcommand\Pscr{\mathscr{P}}

\newcommand\B{\mathbb{B}}
\newcommand\C{\mathbb{C}}

\newcommand\CP{\mathbb{CP}}

\newcommand\N{\mathbb{N}}

\newcommand\R{\mathbb{R}}

\newcommand\Z{\mathbb{Z}}

\newcommand\longhookrightarrow{\ensuremath{\lhook\joinrel\relbar\joinrel\rightarrow}}

\renewcommand\span{\mathrm{span}}

\newcommand\Flux{\mathrm{Flux}}
\newcommand\GCMI{\mathrm{GCMI}}
\newcommand\CMI{\mathrm{CMI}}
\newcommand\NC{\mathrm{NC}}

\newcommand\Hess{\mathrm{Hess}}

\def\span{\mathrm{span}}

\def\Flux{\mathrm{Flux}}

\usepackage{color}
\usepackage[color=blue!20]{todonotes}

\begin{document}


\fancyhead[LO]{Every nonflat conformal minimal surface is homotopic to a proper one}
\fancyhead[RE]{Tja\v{s}a Vrhovnik}
\fancyhead[RO,LE]{\thepage}

\thispagestyle{empty}


\begin{center}

{\bf\Large Every nonflat conformal minimal surface \\  is homotopic to a proper one}

\medskip

%
%
{\bf Tja\v{s}a Vrhovnik}
\end{center}


%
%
\medskip

\begin{quoting}[leftmargin={5mm}]
{\small
\noindent {\bf Abstract}\hspace*{0.1cm}
Given an open Riemann surface $M$, we prove that every nonflat conformal minimal immersion $M\to\R^n$ ($n\geq 3$) is homotopic through nonflat conformal minimal immersions $M\to\R^n$ to a proper one. If $n\geq5$, it may be chosen in addition injective, hence a proper conformal minimal embedding. Prescribing its flux, as a consequence, every nonflat conformal minimal immersion $M\to\R^n$ is homotopic to the real part of a proper holomorphic null embedding $M\to\C^n$. We also obtain a result for a more general family of holomorphic immersions from an open Riemann surface into $\C^n$ directed by Oka cones in $\C^n$.

\noindent{\bf Keywords}\hspace*{0.1cm} 
Riemann surface, minimal surface, proper minimal surface, null curve, directed immersion, Oka manifold

\noindent{\bf Mathematics Subject Classification (2020)}\hspace*{0.1cm} 
53A10, 
32H02, 
32Q56, 
53C42. 
}
\end{quoting}



\section{Introduction}
\label{sec:intro}

Minimal surfaces have been a topic of interest for centuries, however, it has not been until the pioneering work by Osserman~\cite{osserman1986survey} when the systematic research began.
More recently, especially during the last decade, minimal surfaces in Euclidean spaces have been studied extensively, using complex-analytic methods and Oka theory. We refer to the monograph~\cite{alarcon2021minimal} for a detailed description of the development on the topic.
Let $M$ be an open Riemann surface, and let $n\geq 3$ be an integer. A smooth immersion $u=(u_1,\dots,u_n)\colon M\to\R^n$ is \emph{conformal} if and only if the $(1,0)$-differential $\partial{u}=(\partial{u}_1,\dots,\partial{u}_n)$ satisfies the nullity condition
\begin{equation} \label{eq:null-condition}
	(\partial{u}_1)^2 + (\partial{u}_2)^2 + \cdots + (\partial{u}_n)^2 = 0.
\end{equation}
Such an immersion is \emph{minimal} if and only if it is harmonic. 
%
Moreover, a conformal minimal immersion $u\colon M\to\R^n$ is \emph{nonflat} if its image $u(M)\subset\R^n$ is not contained in any affine $2$-plane in $\R^n$.
Nonflatness is a natural condition that enables one to embed the derivative map $\partial{u}/\theta\colon M\to{\bf A}_{*}=\{(z_1,\ldots,z_n)\in\C^n: \sum_{i=1}^n z_{i}^2=0\}\setminus\{0\}$ (where $\theta$ is a fixed nowhere vanishing holomorphic $1$-form on $M$) as the core of a period dominating spray of holomorphic maps from an open Riemann surface $M$ into the Oka manifold ${\bf A}_{*}$ (see, e.g.~\cite[Lemma~3.2.1]{alarcon2021minimal}). On the contrary, flat conformal minimal immersions turn out to be the critical points of the period map~\eqref{eq:period map} and may arise as the singular points of the space of conformal minimal immersions. For that reason, our results, as well as the majority of the results we refer to, examine spaces of nonflat maps.
We equip the space of conformal minimal immersions $M\to\R^n$, denoted by $\CMI(M,\R^n)$, with the compact-open topology, and consider its subspace
\begin{equation*}
	\CMI_{*}(M,\R^n)\subset \CMI(M,\R^n)
\end{equation*}
consisting of all nonflat conformal minimal immersions $M\to\R^n$.

The following notion is closely related to the above. A holomorphic immersion $z=(z_1,\dots,z_n)\colon M\to\C^n$ ($n\geq 3$) from an open Riemann surface $M$ into $\C^n$ is said to be a \emph{holomorphic null curve} (or a \emph{null holomorphic immersion}) if it satisfies the nullity condition $(d{z}_1)^2 + (d{z}_2)^2 + \cdots + (d{z}_n)^2 = 0$~(cf.~\eqref{eq:null-condition}). 
Such curve is \emph{nonflat} if its image $z(M)\subset\C^n$ is not contained in any affine complex line.
We write $\NC(M,\C^n)$ for the space of null curves $M\to\C^n$ endowed with the compact-open topology, and denote by
\begin{equation*}
	\NC_{*}(M,\C^n)\subset \NC(M,\C^n)
\end{equation*}
the subspace consisting of all nonflat holomorphic null curves $M\to\C^n$.
Since the real and the imaginary parts of a null curve $M\to\C^n$ define conformal minimal immersions $M\to\R^n$, we have the following natural inclusions
\begin{equation*}
	\Re\NC(M,\C^n)\longhookrightarrow\CMI(M,\R^n) \quad \text{and} \quad \Re\NC_{*}(M,\C^n)\longhookrightarrow\CMI_{*}(M,\R^n),
\end{equation*}
where $\Re\NC(M,\C^n)=\{\Re z\colon z\in\NC(M,\C^n)\}$ (and analogously for the space of real parts of nonflat holomorphic null curves).

In the last decade or so, one of the goals in the minimal surface theory was to understand the homotopy-theoretic properties of the space $\CMI_{*}(M,\R^n)$. That was studied by Forstneri\v{c} and L{\'a}russon in~\cite{forstneric2019parametric}, as they showed the parametric h-principle with approximation for the inclusion 
\begin{equation} \label{eq:parametric h-inclusion}
	\iota\colon\Re\NC_{*}(M,\C^n) \longhookrightarrow \CMI_{*}(M,\R^n)
\end{equation} 
of the space of real parts of nonflat null holomorphic immersions $M\to\C^n$ into the space of nonflat conformal minimal immersions $M\to\R^n$. More precisely, the inclusion~\eqref{eq:parametric h-inclusion} is a weak homotopy equivalence~\cite[Theorem~1.1]{forstneric2019parametric}, meaning that it induces a bijection of path components of the corresponding spaces and an isomorphism of homotopy groups
\begin{equation*}
	\pi_k \left(\Re \NC_{*}(M,\C^n)\right) \overset{\cong}\longrightarrow \pi_k \left(\CMI_{*}(M,\R^n)\right)
\end{equation*}
for every $k\geq1$ and every choice of base point.
Furthermore, by~\cite[Theorem~1.2]{forstneric2019parametric}, the maps in the commuting diagram
\begin{equation*}
	\begin{tikzcd}
		\NC_{*}(M,\C^n) \arrow{r}{\phi} \arrow[swap]{d}{\Re} & \Oscr_{*}(M,{\bf A}_{*}) \arrow[hookrightarrow]{r} & \Oscr(M,{\bf A}_{*}) \arrow[hookrightarrow]{r} & \Cscr(M,{\bf A}_{*}) \\%
\Re \NC_{*}(M,\C^n) \arrow[hookrightarrow]{r}{\iota} & \CMI_{*}(M,\R^n) \arrow{u}{\psi} 
	\end{tikzcd}
\end{equation*}
are weak homotopy equivalences. Here, $\phi(z)=\partial{z}/\theta$ and $\psi(u)=2\partial{u}/\theta$ (with $\theta$ being a fixed nowhere vanishing holomorphic $1$-form on $M$), and the spaces $\Oscr_{*}(M,{\bf A}_*)$, $\Oscr(M,{\bf A}_*)$, $\Cscr(M,{\bf A}_*)$, of (nonflat) holomorphic, and continuous, respectively, maps $M\to{\bf A}_*$ are endowed with the compact-open topology. These maps are homotopy equivalences when an open Riemann surface $M$ is of finite topological type. This reduces the understanding of the homotopy type of the space $\CMI_{*}(M,\R^n)$ to a purely topological problem: it is the same as the space $\Cscr(M,{\bf A}_*)$.
As a special case of \cite[Theorem~4.1]{forstneric2019parametric}, the parametric h-principle with approximation for the inclusion~\eqref{eq:parametric h-inclusion}, they obtained a result by Alarc\'on and Forstneri\v{c}~\cite[Theorem~1.1]{alarcon2018every}, which says that for every nonflat conformal minimal immersion $u\colon M\to\R^n$ of an open Riemann surface, there exists a smooth isotopy of nonflat conformal minimal immersions $u_t\colon M\to \R^n$, $t\in [0,1]$, such that $u_0=u$ and $u_1$ is the real part of a holomorphic null curve $M\to\C^n$.
Here, by \emph{isotopy}, we mean a regular homotopy of smooth $1$-parameter family of immersions.
Nonetheless, the question whether the inclusion $\Re\NC(M,\C^n) \hookrightarrow \CMI(M,\R^n)$ is also a weak homotopy equivalence remains open.

Going further, a natural question is to study global properties of minimal surfaces, such as completeness and properness.
Recall that an immersion $u\colon M\to\R^n$ is \emph{complete} if the pullback $u^{*}(ds^2)$ of the Euclidean metric $ds^2$ on $\R^n$ is a complete metric on $M$.
It was proven in~\cite[Theorem~7.1]{alarcon2016embedded} that the space $\CMI_{*}^{c}(M,\R^n)$ of complete nonflat conformal minimal immersions $M\to\R^n$ (endowed with the compact-open topology) is a dense subspace of the space $\CMI(M,\R^n)$ of all conformal minimal immersions. Even more, it is a \emph{residual} (or \emph{comeagre}) set in Baire category sense~\cite[Theorem~1.2]{alarcon2024generic}, hence, \emph{almost all} conformal minimal immersions are complete. (We refer e.g. to~\cite{willard1970general} for the background on Baire spaces.)
Concerning the homotopy-theoretic properties of the space $\CMI_{*}^{c}(M,\R^n)$, the first known result, due to Alarc\'on and Forstneri\v{c}~\cite[Theorem~1.2]{alarcon2018every}, affirms that every nonflat conformal minimal immersion $M\to\R^n$ can be deformed through maps of the same type to a complete nonflat conformal minimal immersion. In particular, the inclusion $\CMI_{*}^{c}(M,\R^n) \hookrightarrow \CMI_{*}(M,\R^n)$ induces a surjection of path components.
Just recently, Alarc\'on and L\'arusson in~\cite{alarcon2025strong} obtained a much stronger result, namely, the parametric h-principle for complete nonflat conformal minimal immersions of an open Riemann surface $M$ into $\R^n$. As a consequence, it follows that the inclusion
\begin{equation} \label{eq:parametric h-inclusion complete}
	\iota\colon \CMI_{*}^{c}(M,\R^n) \longhookrightarrow \CMI_{*}(M,\R^n)
\end{equation}
is a weak homotopy equivalence, and moreover, if $M$ is of finite topological type, then the inclusion~\eqref{eq:parametric h-inclusion complete} is a homotopy equivalence. The latter means that there exists a homotopy inverse $\eta\colon \CMI_{*}(M,\R^n) \to \CMI_{*}^{c}(M,\R^n)$ to $\iota$~\eqref{eq:parametric h-inclusion complete} (see~\cite[Theorem~1.1, Corollary~1.2]{alarcon2025strong}).
By the parametric h-principle discussed in the previous paragraph, the spaces $\CMI_{*}^{c}(M,\R^n)$ and $\Cscr(M,{\bf A}_*)$ have the same homotopy type~\cite[Corollary~1.8]{alarcon2025strong}.
In~\cite{alarcon2025strong}, they also obtained analogous results for holomorphic null curves $M\to\C^n$.
Another result in this direction concerns the question which holomorphic maps $M\to{\bf Q}^{n-2}$ are Gauss maps of complete conformal minimal immersions; we refer to~\cite{fujimoto1988onthenumber,osserman1963oncomplete,weitsman1987some} for some important results about this question.
Here, 
\begin{equation*}
	{\bf Q}^{n-2} = \left\{ [z_1:\cdots:z_n]\in \CP^{n-1}\colon z_{1}^2+\cdots+z_{n}^2=0 \right\}
\end{equation*}
denotes the hyperquadric in $\CP^{n-1}$, and the \emph{Gauss map} of a conformal minimal immersion $u\in\CMI(M,\R^n)$ is the holomorphic map $\Gscr(u)\colon M\to{\bf Q}^{n-2}$ given by
\begin{equation*}
	\Gscr(u)(p) = \left[ \partial{u}_1(p):\cdots:\partial{u}_n(p) \right], \quad p\in M;
\end{equation*}
see~\cite{hoffman1980thegeometry}.
Writing $\Oscr_{*}^{c}(M,{\bf Q}^{n-2}) = \Gscr(\CMI_{*}^{c}(M,\R^n))$ for the space of the Gauss maps of complete nonflat conformal minimal immersions, Alarc\'on and L\'arusson proved~\cite[Theorem~1.2]{alarcon2024space} that the inclusion
\begin{equation*}
	\Oscr_{*}^{c}(M,{\bf Q}^{n-2}) \longhookrightarrow \Cscr(M,{\bf Q}^{n-2})
\end{equation*}
is a weak homotopy equivalence, and a homotopy equivalence whenever $M$ has finite topological type. 

As far as we know, no homotopy-theoretic properties of the space of proper nonflat conformal minimal immersions have been studied yet, therefore being the aim of the present paper. 
Recall that a map $x\colon M\to \R^n$ is \emph{proper} if for every compact set $K\subset \R^n$ the preimage $x^{-1}(K)$ is compact in $M$. 
We denote by 
\begin{equation*}
	\CMI_{*}^{p}(M,\R^n) \subset \CMI(M,\R^n)
\end{equation*} 
the subspace of proper nonflat conformal minimal immersions $M\to\R^n$ endowed with the compact-open topology.
Since properness implies completeness, studying spaces of proper minimal surfaces is a natural next step.
However, as opposite to complete nonflat conformal minimal immersions, the subspace $\CMI_{*}^{p}(M,\R^n)$ is a \emph{meagre} set in the space $\CMI(M,\R^n)$ in Baire category sense~\cite[Corollary~1.5]{alarcon2024generic}. This means that, despite being dense~\cite{alarcon2016embedded,alarcon2012minimal}, this subspace is very small, even negligible, so it is much more evolved to find proper maps among all conformal minimal immersions $M\to\R^n$.
Here is our main result.

\begin{theorem} \label{thm:main}
Let $M$ be an open Riemann surface, $u\colon M \to \R^n$ ($n\geq3$) a nonflat conformal minimal immersion and $\mathfrak{p}\colon H_1(M,\Z)\to \R^n$ a group homomorphism. Then there is a continuous family $u_t\colon M\to \R^n$ of nonflat conformal minimal immersions, $t\in [0,1]$, such that $u_0=u$, $u_1$ is proper and $\Flux(u_1)=\mathfrak{p}$.

Moreover, if $n\geq5$, then $u_t$ can be chosen such that $u_1$ is in addition injective, that is, a proper nonflat conformal minimal embedding $M\to\R^n$.
\end{theorem}

For the notion of \emph{flux}, see Definition~\ref{def: flux}.
Since any conformal minimal immersion $u\colon M\to\R^n$ is the real part of a holomorphic null curve $M\to\C^n$ if and only if $u$ has vanishing flux, we have the following direct consequence of Theorem~\ref{thm:main}.

\begin{corollary} \label{cor:homotopic to null curve}
Let $M$ be an open Riemann surface. Every nonflat conformal minimal immersion $M\to\R^n$ ($n\geq3$) is homotopic through nonflat conformal minimal immersions to the real part of a proper holomorphic null curve.
\end{corollary}

To prove the theorem, we were not able to employ the Oka principle for sections of ramified holomorphic maps with Oka fibres, as was the case in the proofs of results concerning approximation and interpolation by proper minimal surfaces and directed holomorphic curves~\cite{alarcon2014null,alarcon2021minimal}. The reason is that there is no version of the \emph{Parametric Oka Property with Approximation and Interpolation}~\cite[Theorem~5.4.4]{forstneric2017stein} for when the codomain of sections depends continuously on the parameter. Instead, our approach follows ideas developed in the parametric setting, in particular~\cite{forstneric2019parametric}, but paying additional attention in order to guarantee properness. 
The main new ingredient in our proof is Lemma~\ref{lemma:hom-2fixed}, which roughly says that, given a nonflat conformal minimal immersion $u\colon K\to\R^n$ from a smoothly bounded compact domain $K$ in an open Riemann surface, there exists a number $\varepsilon>0$, such that if there is a nonflat conformal minimal immersion $\tilde{u}\colon K\to\R^n$ $\varepsilon$-close to $u$ on $K$, then they are homotopic through nonflat conformal minimal immersions, and all the maps in the homotopy are close to the given one on $K$.
Although we use Lemma~\ref{lemma:hom-2fixed} with the aim to construct a proper map, it has an implication in a rather different direction which we explain in Section~\ref{sec:bordered}.
Furthermore, in none of the proofs we use the special geometry of the null quadric~\eqref{eq:null quadric}, which enables us to extend the results to a more general family of holomorphic immersions directed by Oka cones. This is treated in Section~\ref{sec:directed}.

We also have the following two corollaries, which provide information on the topology of the space of proper nonflat conformal minimal immersions.

\begin{corollary} \label{cor:surjection CMI}
The inclusion
\begin{equation} \label{eq:whe CMI}
	\CMI_{*}^{p}(M,\R^n) \longhookrightarrow \CMI_{*}(M,\R^n)
\end{equation} 
of the space of proper nonflat conformal minimal immersions $M\to\R^n$ into the space of nonflat conformal minimal immersions $M\to\R^n$ induces a surjection of path components.
\end{corollary}

\begin{corollary} \label{cor:connected components}
Let $M$ be an open Riemann surface and $H_1(M,\Z) \cong \Z^l$ with $l\in\Z_{+}\cup \{\infty\}$.
Then the space $\CMI_{*}^{p}(M,\R^3)$ of proper nonflat conformal minimal immersions $M\to\R^3$ has at least $2^l$ connected components.
\end{corollary}

The second statement is obtained in view of~\cite[Proposition~8.4]{alarcon2018every} and the discussion therein.

\begin{problem} \label{prob:whe CMI}
Does the inclusion~\eqref{eq:whe CMI} induce a bijection of path components? 
Is it a weak homotopy equivalence?
\end{problem}

\begin{remark}
Given Stein manifolds $X$ and $Y$ with $\dim{Y}>2\dim{X}$, every continuous map $X\to Y$ is homotopic to a proper holomorphic embedding. More precisely, by~\cite[Theorem~3.4]{forstneric2018holomorphic}, every Stein manifold with the (volume) density property is strongly universal for proper holomorphic embeddings and immersions. (For the terminology in the theorem, we refer to~\cite{forstneric2018holomorphic} and will not explain it here.) The tools used in the proof of this result are different to ours. One of the reasons why we cannot follow the same scheme is that instead of working with maps, we consider their derivatives, as we have to carefully control the periods of the derivative maps, which is a main difficulty when dealing with conformal minimal immersions.
\end{remark}

\subsection*{Organization of the paper}
We introduce the relevant definitions and notation in Section~\ref{sec:prelim}. Sections~\ref{sec:main lemma} and~\ref{sec:(non)critical} are the core of the paper. In particular, in Section~\ref{sec:main lemma}, we develop the main new technical result, Lemma~\ref{lemma:hom-2fixed}, whereas Section~\ref{sec:(non)critical} provides the noncritical and critical cases in the proof of Theorem~\ref{thm:main}; see Proposition~\ref{prop:noncritical} and Lemma~\ref{lemma:critical}. The main result, Theorem~\ref{thm:main}, is proven in Section~\ref{sec:proof-main} by a standard inductive process. Section~\ref{sec:bordered} treats minimal surfaces defined on bordered Riemann surfaces, upgrading some already known approximation results.
Finally, in Section~\ref{sec:directed}, we focus on nondegenerate directed holomorphic immersions, and by slightly adapting methods developed in the former sections, show Theorem~\ref{thm:directed}, that generalizes Theorem~\ref{thm:main}.


\section{Notation and preliminaries}
\label{sec:prelim}

We begin the preparatory section by establishing notation and recalling some basic definitions.
Denote by $\N=\{1,2,\ldots\}$ the set of natural numbers, by $\Z_+=\N\cup\{0\}$ the set of non-negative integers, and assume that $n\in\N$, $n\geq 3$. 
Let $|\cdot|$ denote the Euclidean norm in $\R^n$, and $||\cdot||_{\infty}$ the infinity norm in $\R^n$.
Let $K$ be a compact topological space and let $f:K\to\R^n$ be a continuous map. We write $\|f\|_K$ for the maximum norm of $f$ on $K$. 
Given maps $f,g\colon K\to\R^n$, we say that $f$ approximates $g$ on $K$, and write $f \approx g$, if $||f-g||_{K}<\varepsilon$ for some $\varepsilon>0$ which is sufficiently small according to the argument.
If not specified otherwise, throughout the paper, we assume that surfaces are connected.

Let $X,\ Y$ be complex manifolds and $M \subset X$ a subset of $X$. By $\mathscr{C}^r(M,Y)$, $r\in\N\cup\{0\}$, we denote the space of $r$-times continuously differentiable maps $M\to Y$, and by $\mathscr{O}(M,Y)$ the space of holomorphic maps from some unspecified neighbourhood of $M$ in $X$ to $Y$.
If, in addition, we are given a closed subset $K \subset X$, then we write $\mathscr{A}(K,Y)$ for the space of continuous maps $K\to Y$ which are holomorphic in the interior $\mathring K \subset K$ of $K$, i.e., $\mathscr{A}(K,Y) = \mathscr{C}^0(K,Y) \cap \mathscr{O}(\mathring K,Y)$.
In the case when $Y=\C$, we omit writing the latter in the corresponding spaces of functions, and therefore consider the spaces $\mathscr{C}^r(M)$, $\mathscr{O}(M)$, and $\mathscr{A}(K)$.

Let $M$ be an open Riemann surface and let $u=(u_1,\dots,u_n)\colon M\to\R^n$ ($n\geq 3$) be a conformal immersion of class $\Cscr^2$. Recall that the following statements are equivalent~\cite[Theorem~2.3.1]{alarcon2021minimal}. 
\begin{itemize}
\item[\rm (i)] $u$ is minimal.
\item[\rm (ii)] $u$ is harmonic, i.e., $dd^cu=0$.
\item[\rm (iii)] $u$ has vanishing mean curvature vector field, i.e., ${\bf H}=0$.
\item[\rm (iv)] the $\C^n$-valued $(1,0)$-form $\partial{u}=(\partial{u}_1,\dots,\partial{u}_n)$ is holomorphic and satisfies the nullity condition~\eqref{eq:null-condition}.
\item[\rm (v)] Let $\theta$ be a nowhere vanishing holomorphic $1$-form on $M$ (such exists by the Oka-Grauert principle~\cite{grauert1958analytische,grauert1957approximation,grauert1957holomorphe}; see also~\cite{gunning1967immersion}). Then the map 
	\begin{equation*}
		f=(f_1,\dots,f_n)=2\partial{u}/\theta\colon M\longrightarrow\C^n
	\end{equation*}
is holomorphic and has range in the punctured null quadric
	\begin{equation} \label{eq:null quadric}
		{\bf A}_{*}:={\bf A}\setminus\{0\}=\left\{z=(z_1,\ldots,z_n)\in\C^n: z_1^2+\cdots+z_n^2=0\right\}\setminus\{0\}\subset\C^n.
	\end{equation}
\end{itemize}

Minimal surfaces appear as the real and the imaginary parts of \emph{holomorphic null curves}, that is, holomorphic immersions $z=(z_1,\dots,z_n)\colon M\to\C^n$ ($n\geq 3$) from an open Riemann surface $M$ into $\C^n$ that satisfy $\sum_{i=1}^{n}(dz_{i})^2=0$.
Conversely, given a simply connected domain $D\subset M$ and a conformal minimal immersion $u\colon M\to\R^n$, the restriction $u|_{D}\colon D\to\R^n$ is the real part of a holomorphic null curve $z\colon D\to\C^n$.
Moreover, $u\colon M\to\R^n$ globally corresponds to the real part of a holomorphic null curve if and only if it holds that $\int_{C}d^{c}u=0$ for any closed curve $C\subset M$.
In particular, if the latter is satisfied, i.e., the conjugate differential $d^{c}u=2\Im(\partial{u})$ is an exact $1$-form on $M$, then $u$ admits a harmonic conjugate $v$, and $z=u+\imath v\colon M\to\C^n$ is a holomorphic null curve (here, $\imath^2=-1$ denotes the imaginary unit).

\begin{definition} \label{def: flux}
The \emph{flux} of a conformal minimal immersion $u\colon M\to\R^n$ is the group homomorphism $\Flux(u)\colon H_1(M,\Z)\to\R^n$ on the first homology group of $M$, defined by
\begin{equation*}
	\Flux(u) \left([C]\right) = \int_{C} 2\Im(\partial{u}) = \int_{C} d^c{u}, \quad [C]\in H_1(M,\Z).
\end{equation*}
\end{definition}

Since the above integral depends only on the homology class of a path, we will write $\Flux(u)(C)=\Flux(u)([C])$ for any closed curve $C\subset M$.
We now define a special class of \emph{generalized conformal minimal immersions} on an admissible set in an open Riemann surface. These sets appear naturally in the handlebody structure of Riemann surfaces, and consequently, in the inductive process in the proof of our main result.

\begin{definition} \label{def: admissible set}
Given an open Riemann surface $M$, an \emph{admissible set} in $M$ is a compact set $S=K\cup E$, where $K$ is a finite (or possibly empty) union of pairwise disjoint compact domains with piecewise smooth boundaries in $M$ and $E = S\setminus \mathring K$ is a finite (or possibly empty) union of pairwise disjoint smooth Jordan arcs and closed Jordan curves intersecting $K$ only at their endpoints (or not at all) and such that their intersections with the boundary $bK$ of $K$ are transverse.
\end{definition}

\begin{definition}[Generalized conformal minimal immersion] \label{def: GCMI}
Let $S=K\cup E$ be an admissible set in a Riemann surface $M$, $\theta$ be a nowhere vanishing holomorphic $1$-form on $M$ and $n\geq 3$.
A \emph{generalized conformal minimal immersion} $S\to\R^n$ is a pair $(u,f\theta)$, where $u\in\Cscr(S,\R^n)$ is such that its restriction to $\mathring S=\mathring K$ is a conformal minimal immersion, and $f\in\Ascr(S, {\bf A}_*)$ satisfies
\begin{itemize}
	\item[\rm (i)] $f\theta=2\partial{u}$ on $K$, and
	\item[\rm (ii)] for any smooth path $\alpha$ in $M$ that parametrizes a connected component of $E=\overline{S\setminus K}$, it holds that $\Re(\alpha^*(f\theta))=\alpha^*(du)=d(u\circ\alpha)$.
\end{itemize}
\end{definition}

Given an admissible set $S$, we denote the set of generalized conformal minimal immersions $S\to\R^n$ by $\GCMI(S,\R^n)$. 
The notion of nonflatness applies to this case in the following way: $(u,f\theta)\in\GCMI(S,\R^n)$ is \emph{nonflat} if and only if the map $f\in\Ascr(S,{\bf A}_*)$ is nonflat on every relatively open subset of $S$. We write
\begin{equation*}
	\GCMI_{*}(S,\R^n) \subset \GCMI(S,\R^n)
\end{equation*}
for the subset consisting of all nonflat generalized conformal minimal immersions $S\to\R^n$.
Sometimes, we will simply denote by $u\in\GCMI(S,\R^n)$ for a generalized conformal minimal immersion $u\colon S\to\R^n$, when the corresponding holomorphic $1$-form $f\theta$ can be derived from the context.

We call a compact set $K$ in a Riemann surface $M$ \emph{holomorphically convex} or a \emph{Runge set} in $M$ if and only if $K$ has no holes in $M$, which is satisfied precisely when every function $f\in\Oscr(K)$ is the uniform limit on $K$ of functions in $\Oscr(M)$ (see~\cite[Lemma~1.12.3]{alarcon2021minimal}). Sets of this type are important as they satisfy the Runge approximation theorem.

A complex manifold $X$ is an \emph{Oka manifold} if every holomorphic map $K \to X$ from a neighbourhood of a compact convex set $K \subset \C^n$ can be approximated uniformly on $K$ by entire maps $\C^n \to X$ (see~\cite[Definition~5.4.1]{forstneric2017stein}.)


\section{Nearby conformal minimal surfaces are isotopic}
\label{sec:main lemma}

This section is devoted to the main new ingredient, Lemma~\ref{lemma:hom-2fixed}. It says that given a nonflat conformal minimal immersion $u\colon K\to\R^n$ defined on a smoothly bounded compact domain $K$, if there is another nonflat conformal minimal immersion $\tilde{u}\colon K\to\R^n$ uniformly close to $u$ on $K$, then they are homotopic through nonflat conformal minimal immersions. This somehow upgrades the property of being close to being homotopic, which is one of the major features we employ in the proof of the main result. 
As explained in Section~\ref{sec:bordered}, it enables us to improve some other results concerning approximation of minimal surfaces defined on compact bordered Riemann surfaces, and we expect to establish further applications.

\begin{lemma} \label{lemma:hom-2fixed}
Let $K$ be a smoothly bounded compact domain in an open Riemann surface $M$, $u\in\CMI_{*}(K,\R^n)$, $x_0\in \mathring K$ and $\mu>0$.
There is an $\varepsilon>0$ such that if $\tilde{u}\in\CMI_{*}(K,\R^n)$ with $|\tilde{u}-u|<\varepsilon$ on $K$ and $u(x_0)=\tilde{u}(x_0)$, then there exists a homotopy of nonflat conformal minimal immersions $u_t\colon K\to \R^n$, $t\in [0,1]$, satisfying the following.
\begin{itemize}
\item[\rm (i)] $u_0=u$.
\item[\rm (ii)] $u_1=\tilde{u}$.
\item[\rm (iii)] $|u_t-u|<\mu$ on $K$ for all $t\in [0,1]$.
\item[\rm (iv)] $u_t(x_0)=u(x_0)$ for all $t\in [0,1]$.
\end{itemize}
\end{lemma}

\begin{proof}
Assume that $u\in\CMI_{*}(K,\R^n)$ is as in the statement of the lemma. 
Fix a nowhere vanishing holomorphic $1$-form $\theta$ on $M$~\cite{gunning1967immersion}. 
Pick a family of smooth oriented embedded arcs and closed Jordan curves $\Cscr=\{C_1,\dots,C_l\}$ in $M$ that forms a Runge homology basis of $H_1(K,\Z)$, i.e., the union $C=\bigcup_{i=1}^{l}C_i$ is a Runge subset of $K$. We associate to $\Cscr$ the \emph{period map}
\begin{eqnarray} \label{eq:period map}
	\Pscr &=& \left(\Pscr_1,\dots,\Pscr_l\right) \colon \Ascr(K,{\bf A}_*) \longrightarrow \left(\C^n\right)^l, \nonumber \\
	\Pscr_i(h) &=& \int_{C_i}h\theta \in \C^n, \quad h\in\Ascr(K,{\bf A}_*), \ i=1,\dots,l.
\end{eqnarray}
Consider the nonflat holomorphic map
\begin{equation*}
	f:=2\partial u/\theta\colon K\longrightarrow{\bf A}_*;
\end{equation*}
recall that $u$ is nonflat by assumption.
For every $i=1,\dots,l$, pick points $p_{i,k}\in C_i$ and holomorphic vector fields $V_{i,k}$ on $\C^n$, $k=1,\dots,n$, tangent to ${\bf A}$ such that $\span\{V_{i,k}(f(p_{i,k}))\colon 1\leq k\leq n\} = \C^n$. Denote by $\phi_{\zeta}^{i,k}$ the flow of $V_{i,k}$, and write $\zeta=(\zeta_1,\dots,\zeta_l)\in (\C^n)^l$ with $\zeta_i=(\zeta_{i,1},\dots,\zeta_{i,n})\in\C^n$.
Then, by~\cite[Lemma~3.2.1]{alarcon2021minimal}, for every pair $(i,k)$, we can choose a smooth function $h_{i,k}\colon C\to\C$ supported on a short arc around the point $p_{i,k}\in C_i$ in such a way that the nonflat holomorphic map $f$ embeds as the core $f^0=f$ of the \emph{period dominating spray}
\begin{equation} \label{eq:def-spray}
	f^{\zeta}(p) = \phi_{h_{1,1}(p)\zeta_{1,1}}^{1,1}\circ\dots\circ \phi_{h_{l,n}(p)\zeta_{l,n}}^{l,n} \left(f(p)\right)\in{\bf A}_*,
\end{equation} 
that is well-defined for $p\in C$ and $\zeta\in U$, where $0\in U\subset\C^{ln}$ is a neighbourhood of the origin.
The period dominating property means that the differential
\begin{equation*}
	U\ni\zeta \longmapsto \frac{d}{d\zeta}\Big|_{\zeta=0} \Pscr \left(f^{\zeta}\right)=\frac{d}{d\zeta}\Big|_{\zeta=0} \left(\int_{C_i}f^{\zeta}\theta \right)_{i=1,\dots,l} \in \C^{ln}
\end{equation*}
is surjective.

Let us assume that we are given another map $\tilde{f}\in\Oscr(K,{\bf A}_*)$ close to $f$ on $K$. We construct a homotopy of nonflat holomorphic maps $f_t$ ($t\in[0,1]$) with $f_0=f$ and $f_1=\tilde{f}$, that is close to $f$ on $K$ for all $t\in[0,1]$.
More precisely, we claim that there exists a number $\varepsilon_1>0$ satisfying the following: if $\tilde{f}\colon K\to{\bf A}_*$ is a nonflat holomorphic map such that $|\tilde{f}-f|<\delta$ on $K$ for some number $0<\delta\leq\varepsilon_1$, then there exists a homotopy $f_t\colon K\to{\bf A}_*$ of nonflat holomorphic maps such that $f_0=f$, $f_1=\tilde{f}$ and $|f_t-f|<\delta$ on $K$ for every $t\in [0,1]$.
Indeed, by Siu~\cite{siu1976every} and Grauert's \emph{tubular neighbourhood theorem}~\cite{docquier1960levisches} (see also~\cite[Theorem~3.3.3]{forstneric2017stein}), the graph of $f$ over a small neighbourhood $W$ of $K$ (on which $f$ and $\tilde{f}$ are holomorphic) has an open Stein neighbourhood in $W\times{\bf A}_*$, which is biholomorphic to an open neighbourhood of the zero section in the normal bundle to the graph, and there is a homotopy of holomorphic maps in this Stein neighbourhood. The claim then follows.

We choose the number $\varepsilon_1>0$ so small, such that if $\tilde{f}\in\Oscr(K,{\bf A}_*)$ is as above and the holomorphic $1$-forms $f_t\theta$ have vanishing real periods for $t\in [0,1]$, i.e., $\Re\Pscr(f_t)=0$ (see~\eqref{eq:period map}), then 
\begin{equation} \label{eq:lemma-u_t} 
	u_t(x) := u(x_0) + \Re\int_{x_0}^{x} f_t\theta, \quad x\in K, \ t\in[0,1],
\end{equation}
defines a homotopy of nonflat conformal minimal immersions $K\to\R^n$ that satisfy $|u_t-u|<\mu$ on $K$ for every $t\in [0,1]$. (Note that $u_t(x_0)=u(x_0)$, $u_0=u$.)

Now, assuming that we have $f$ and the homotopy $f_t$ from above, we shall correct the real periods of holomorphic $1$-forms $f_t\theta$ to zero for all $t\in[0,1]$. (Observe that at the moment, only $\Re\Pscr(f_0)=0$.)
Since $f_t$ are nonflat on $K$, $t\in[0,1]$, we can replace $f$ by $f_t$ as the core of the period dominating spray~\eqref{eq:def-spray} of nonflat holomorphic maps
\begin{equation*}
	f_{t}^{\zeta}\colon K\longrightarrow{\bf A}_*, \quad \zeta\in\B, \ t\in[0,1],
\end{equation*}
depending holomorphically on a parameter $\zeta$ in a ball $0\in\B\subset\C^{ln}$. That is, $f_{t}^0=f_t$ and the differential of the period map $\zeta \mapsto \Pscr(f_{t}^{\zeta})$ at $\zeta=0$ is surjective for every $t\in[0,1]$.

There exists a number $\varepsilon_2\in(0,\varepsilon_1]$ such that the following is satisfied: if $f_t\colon K\to {\bf A}_*$, $t\in[0,1]$, is a homotopy of nonflat holomorphic maps such that $f_0=f$, $|f_t-f|<\varepsilon_2$ on $K$ ($t\in[0,1]$) and $\Re\Pscr(f_1)=0$, then the implicit function theorem furnishes us with a continuous map
\begin{equation} \label{eq:lemma-zeta}
	\zeta\colon[0,1]\longrightarrow r\B\subset\C^{ln}, \quad \zeta(0)=\zeta(1)=0,
\end{equation}
where $r\in(1/2,1)$, such that the homotopy of nonflat holomorphic maps
\begin{equation*}
	g_t := f_{t}^{\zeta(t)} \colon K\longrightarrow{\bf A}_*, \quad t\in[0,1],
\end{equation*}
satisfies that $\Re\Pscr(g_t)=\Re\Pscr(f_0)=\Re\Pscr(f)=0$, $|g_t-f|<\varepsilon_1$ on $K$ for all $t\in[0,1]$, and $g_0=f$, $g_1=f_1$ (see~\eqref{eq:lemma-zeta}).

Taking into account that $K$ is compact and using the Cauchy estimates, we can find a number $\varepsilon>0$ such that if $\tilde{u}\in\CMI_{*}(K,\R^n)$ is a nonflat conformal minimal immersion that satisfies $|\tilde{u}-u|<\varepsilon$ on $K$, then $|2\partial\tilde{u}/\theta-f|<\varepsilon_2$. By the above arguments, assuming also that $\tilde{u}(x_0)=u(x_0)$, and replacing $f_t$ by $g_t$ in the definition of $u_t$~\eqref{eq:lemma-u_t}, lemma holds for this choice of $\varepsilon$.
\end{proof}

\begin{remark} \label{rem:lemma-hom2fixed}
The assumption in Lemma~\ref{lemma:hom-2fixed} that the values of conformal minimal immersions $u$ and $\tilde{u}$ coincide at the point $x_0\in\mathring K$ is purely technical, as it ensures that $u_1=\tilde{u}$ (see~\eqref{eq:lemma-u_t}), giving also property~(iv). However, this assumption is not necessary in our applications. In particular, by a slight translation of $\tilde{u}$ (note that $\tilde{u}\approx u$ on $K$), we may assume that the condition is satisfied, and then apply Lemma~\ref{lemma:hom-2fixed} to $u$ and translated $\tilde{u}$, omitting property~(iv). 
\end{remark}


\section{The noncritical and critical cases in the induction}
\label{sec:(non)critical}

In this section, we establish further technical results, which are used in the proof of Theorem~\ref{thm:main}.
As we explain in Section~\ref{sec:directed}, they may be proven in a more general setting of directed immersions from an open Riemann surface into $\C^n$.
We begin with the following proposition that corresponds to the noncritical case in the proof of Theorem~\ref{thm:main}, that is, when the topology of the domains in the inductive process remains the same.

\begin{proposition} \label{prop:noncritical}
Let $K$ and $L$ be smoothly bounded compact domains in an open Riemann surface $M$ such that $K\subset\mathring L$ and $K$ is a strong deformation retract of $L$. 
Let $u_p\in\CMI_{*}(K,\R^n)$ be a continuous family of nonflat conformal minimal immersions for $p\in [0,1]$, $u_0\in \CMI_{*}(L,\R^n)$, and assume that $u_1$ maps $bK$ into $\R^n\setminus \{0\}$. Pick a number 
\begin{equation} \label{eq:delta}
	0<\delta< \min\left\{||u_1(x)||_{\infty}\colon x\in bK\right\}
\end{equation}
and let $\varepsilon>0$.
There exists a continuous family $\hat{u}_p \in \CMI_{*}(L,\R^n)$, $p\in[0,1]$, satisfying the following conditions.
\begin{itemize}
\item[\rm (i)] $|\hat{u}_p-u_p|<\varepsilon$ on $K$ for every $p\in[0,1]$.

\item[\rm (ii)] $\hat{u}_0=u_0$ on $L$.

\item[\rm (iii)] $||\hat{u}_1(x)||_{\infty}>\delta$ for all $x\in L\setminus\mathring K$ and $||\hat{u}_1(x)||_{\infty}>1/\varepsilon$ for all $x\in bL$.

\item[\rm (iv)] $\Flux(\hat{u}_1) = \Flux(u_1)$.

\item[\rm (v)] $\hat{u}_1$ is injective if $n\geq 5$.
\end{itemize}
\end{proposition}

\begin{proof}
Pick a number $\varepsilon_1\in(0,\varepsilon/3)$.
Note that $u_1\in\CMI_{*}(K,\R^n)$ maps $bK$ into $\R^n\setminus\{0\}$~\eqref{eq:delta}. Hence, by~\cite[Lemma~3.11.1,~Theorem~3.4.1]{alarcon2021minimal}, we may approximate $u_1$ uniformly on $K$ by $u_2\in\CMI_{*}(L,\R^n)$ such that $||u_2(x)||_{\infty}>\delta$ for all $x\in L\setminus\mathring K$, $||u_2(x)||_{\infty}>1/\varepsilon$ for all $x\in bL$, $\Flux(u_2)=\Flux(u_1)$, and $u_2$ is injective if $n\geq 5$.
Further, Lemma~\ref{lemma:hom-2fixed}, applied to $(u,\tilde{u})=(u_1,u_2)$, provides a homotopy $u_p\in\CMI_{*}(K,\R^n)$, $p\in[0,2]$, with the following properties:
\begin{itemize}
	\item[\rm (i.1)] $u_p$ agrees with the given one for every $p\in[0,1]$;
	\item[\rm (ii.1)] $u_0\in\CMI_{*}(L,\R^n)$;
	\item[\rm (iii.1)] $u_2\in\CMI_{*}(L,\R^n)$, $||u_2(x)||_{\infty}>\delta$ for all $x\in L\setminus\mathring K$, $||u_2(x)||_{\infty}>1/\varepsilon$ for all $x\in bL$, $\Flux(u_2)=\Flux(u_1)$, $u_2$ is injective if $n\geq 5$;
	\item[\rm (iv.1)] $|u_p - u_1|<\varepsilon_1$ on $K$ for every $p\in[1,2]$.
\end{itemize}
In the next step, we shall approximate $u_p\in\CMI_{*}(K,\R^n)$, $p\in[0,2]$, uniformly on $K$ by another homotopy of nonflat conformal minimal immersions $L\to\R^n$, fixing those at parameters $p\in\{0,2\}$ (observe that, by (ii.1)--(iii.1), $u_0$ and $u_2$ are already defined on $L$). This is established with the following claim.
%
\begin{claim} \label{claim:2fixed}
There exists a homotopy $\tilde{u}_p\in\CMI_{*}(L,\R^n)$, $p\in[0,2]$, of nonflat conformal minimal immersions that uniformly approximate $u_p$ on $K$ for all $p\in[0,2]$, and such that $\tilde{u}_0=u_0$ and $\tilde{u}_2=u_2$ holds on $L$.
\end{claim}
%
\begin{proof}
The proof uses the same line of arguments as the proof of~\cite[Theorem~4.1, p.~21]{forstneric2019parametric}.
Denote by $P=[0,2]$, and recall that $K\subset L$ are smoothly bounded compact domains in $M$ such that $K$ is a strong deformation retract of $L$, $u_p\in\CMI_{*}(K,\R^n)$ ($p\in P$) and $u_0, u_2\in\CMI_{*}(L,\R^n)$, meeting conditions (i.1)--(iv.1).
Pick a number $\eta\in (0,\varepsilon_1)$.
We shall find a continuous family $\tilde{u}_{p}\in\CMI_{*}(L,\R^n)$, $p\in P$, such that
\begin{itemize}
	\item[\rm (c1)] $\tilde{u}_{p} = u_p$ on $L$ for every $p\in\{0,2\}$;
	\item[\rm (c2)] $|\tilde{u}_{p} - u_{p}|<\eta$ on $K$ for every $p\in P$.
\end{itemize}
The main argument in the construction of such family is the parametric Oka property with approximation in order to approximate period dominating sprays of derivative maps of the homotopy $u_{p}$. We now explain the details.
Fix a nowhere vanishing holomorphic $1$-form $\theta$ on $M$ and choose a Runge homology basis $\Cscr=\{C_1,\dots,C_l\}$ of $H_1(K,\Z)$ (see the proof of Lemma~\ref{lemma:hom-2fixed}), which then also forms a homology basis of $H_1(L,\Z)$. Denote by $\Pscr$ the period map associated to $\Cscr$~\eqref{eq:period map}.
Consider the family of nonflat holomorphic maps 
\begin{equation*}
	f_p := 2\partial{u_p}/\theta \colon K\longrightarrow{\bf A}_*, \quad p\in P,
\end{equation*}
with $f_0$ and $f_2$ extending to the whole $L$ (as we may in view of~(ii.1)--(iii.1)).
Moreover, observe that $\Re\Pscr(f_{p})=0$ for every $p\in P$.
Since all $u_{p}$ are nonflat on $K$, we may embed $f_{p}$ as the core $f_{p}=f_{p,0}$ of a period dominating spray of holomorphic maps
\begin{equation*}
	f_{p,\zeta} \colon K\longrightarrow{\bf A}_*, \quad (p,\zeta) \in P\times\B,
\end{equation*}
that depends holomorphically on the parameter $\zeta\in\B\subset\C^N$ in a ball $0\in\B\subset\C^N$ for some $N\in\N$. In particular, 
\begin{equation*}
	\B\ni\zeta \longmapsto \frac{d}{d\zeta}\Big|_{\zeta=0} \Pscr \left(f_{p,\zeta}\right)=\frac{d}{d\zeta}\Big|_{\zeta=0} \left(\int_{C_i}f_{p,\zeta}\theta \right)_{i=1,\dots,l} \in \C^{ln}
\end{equation*}
is surjective for every $p\in P$.
Furthermore, as ${\bf A}_*$ is an Oka manifold and $K$ is a strong deformation retract of $L$, we may use~\cite[Theorem~5.4.4]{forstneric2017stein} to approximate $f_{p,\zeta}\colon K\to{\bf A}_*$ uniformly on $K$ and uniformly w.r.t. $(p,\zeta)$ by a holomorphic spray of maps
\begin{equation*}
	g_{p,\zeta} \colon L\longrightarrow{\bf A}_*, \quad (p,\zeta)\in P\times r\B,
\end{equation*}
for some $r\in (1/2,1)$, and such that it satisfies $g_{p,0}=f_p$ for $p\in\{0,2\}$.
Provided that approximation of $f_{p,\zeta}$ by $g_{p,\zeta}$ is sufficiently close on $K$, the implicit function theorem furnishes us with a continuous map
\begin{equation} \label{eq:zeta}
	\zeta \colon P \longrightarrow r\B\subset\C^N, \quad \zeta(0)=\zeta(2)=0,
\end{equation}
such that $\tilde{f}_{p} := g_{p,\zeta(p)}\colon L\to {\bf A}_*$ satisfies that $\Pscr(\tilde{f}_{p})=\Pscr(f_{p})$, $p\in P$. 
In particular, $\Re\Pscr(\tilde{f}_{p})=0$ for every $p\in P$. 
Moreover, $\tilde{f}_{p}=f_p$ on $L$ for $p\in\{0,2\}$, as $\zeta$ vanishes for these values~\eqref{eq:zeta}, and recall also that $f_0, f_2$ are defined on $L$. 
Pick a point $x_0\in K$ and consider 
\begin{equation*}
	\tilde{u}_{p}(x) := u_{p}(x_0) + \Re\int_{x_0}^{x} \tilde{f}_{p}\theta, \quad x\in L, \ p\in P.
\end{equation*}
By properties of $\tilde{f}_{p}$, the above integral is well-defined, and $\tilde{u}_{p}\in\CMI_{*}(L,\R^n)$ uniformly approximates $u_p$ on $K$ for every $p\in P$ and satisfies that $\tilde{u}_{p}=u_p$ on $L$ for $p\in\{0,2\}$.
\end{proof}
%
To finish the proof, we shall reparametrize the homotopy $\tilde{u}_p\in\CMI_{*}(L,\R^n)$, $p\in[0,2]$, to obtain a homotopy $\hat{u}_p\in\CMI_{*}(L,\R^n)$, $p\in[0,1]$, meeting the conclusions of the proposition.
There exists a small number $\tau>0$ close to $0$, such that 
\begin{equation} \label{eq:prop-tau}
	\left|u_q - u_r\right|<\varepsilon_1 \quad \text{on } K \text{ for every } q,r\in[1-\tau,1].
\end{equation} 
Define a diffeomorphism $\psi\colon [0,2]\to[0,1]$ with the following properties:
\begin{equation} \label{eq:diffeo-psi}
	\begin{cases}
		\psi(0)=0, \\
		\psi\colon [0,1]\to [0,1-\tau], \\
		\psi\colon [1,2]\to [1-\tau,1].
	\end{cases}
\end{equation}
Let $\phi:=\psi^{-1}\colon [0,1]\to [0,2]$ be its inverse; note that $\phi(0)=0$.
We may choose $\tau>0$ so small and $\psi|_{[0,1]}$ sufficiently close to identity on $[0,1]$, such that, setting 
\begin{equation*}
\hat{u}_p:=\tilde{u}_{\phi(p)}\colon L\longrightarrow\R^n, \quad p\in [0,1], 
\end{equation*}
we obtain a homotopy of nonflat conformal minimal immersions $L\to\R^n$ satisfying the conclusions of the proposition.
Indeed, $\hat{u}_0=\tilde{u}_0=u_0$, $\hat{u}_1=\tilde{u}_2=u_2$, so $||\hat{u}_1(x)||_{\infty}=||u_2(x)||_{\infty}>\delta$ for all $x\in L\setminus\mathring K$, $||\hat{u}_1(x)||_{\infty}>1/\varepsilon$ for all $x\in bL$, $\Flux(\hat{u}_1)=\Flux(u_2)=\Flux(u_1)$, and $\hat{u}_1$ is injective if $n\geq 5$.
For the uniform approximation of $u_p$ by $\hat{u}_p$ on $K$ we argue as follows.
\begin{itemize}
	\item If $p\in[0,1-\tau]$, then $|\hat{u}_p-u_p|<\varepsilon$ on $K$ by~(c2), the choice of numbers $\tau\approx 0$ and $\eta<\varepsilon_1<\varepsilon$, and the definition of $\psi$~\eqref{eq:diffeo-psi}.
	\item If $p\in[1-\tau,1]$, then $|\hat{u}_p-u_p|<|\hat{u}_p-u_{\phi(p)}|+|u_{\phi(p)}-u_1|+|u_1-u_p|<\eta + 2\varepsilon_1 < 3\varepsilon_1 < \varepsilon$ on $K$ by~(c2),~(iv.1) and~\eqref{eq:prop-tau}.
\end{itemize}
This closes the proof of the proposition.
\end{proof}

In the remaining part of this section, we show the following lemma, which is the core of the critical case in the proof of Theorem~\ref{thm:main}. It enables us to extend the homotopy of derivative maps of nonflat conformal minimal immersions from a compact set in an open Riemann surface to a smoothly embedded arc, providing a homotopy of nonflat generalized conformal minimal immersions. Additionally, we control the flux and the values of the conformal minimal immersion corresponding to time $t=1$. The latter guarantees properness, which is the main point.

\begin{lemma} \label{lemma:critical}
Let us assume the following conditions:
\begin{itemize}
\item $M$ is an open Riemann surface, $\theta$ is a nowhere vanishing holomorphic $1$-form on $M$.
\item $K$ is a compact subset of $M$.
\item $S=K\cup E$ is an admissible set in $M$ such that $E$ closes inside $K$ to a closed Jordan curve $C$ with $E=\overline{C\setminus K}$ and $K\cap E=\{x_0,x_1\}$.
\item $\mathfrak{p}\colon H_1(M,\Z)\to\R^n$ is a group homomorphism.
\item $u_0=u\in\CMI_{*}(M,\R^n)$ is a nonflat conformal minimal immersion and $f_0=f=2\partial{u}/\theta\colon M\to {\bf A}_*$ is a nonflat holomorphic map.
\item $u_p\in\CMI_{*}(K,\R^n)$ is a continuous family of nonflat conformal minimal immersions and $f_p=2\partial{u}_p/\theta\colon K\to {\bf A}_*$ is a continuous family of nonflat holomorphic maps for every $p\in P=[0,1]$.
\item $\Omega\subset\R^n$ is a domain containing $0\in\R^n$ and the point $u_1(x_1)-u_1(x_0)$.
\end{itemize}
Then there exist continuous families of maps $\tilde{f}_p\colon S\to {\bf A}_*$ and $\tilde{u}_p\colon S\to\R^n$ ($p\in P$) satisfying the following.
\begin{itemize}
\item[\rm (i)] $\tilde{u}_p=u_p$ and $\tilde{f}_p=f_p$ on $K$ for every $p\in P$.
\item[\rm (ii)] $\tilde{u}_0=u$ and $\tilde{f}_0=f$ on $S$.
\item[\rm (iii)] $(\tilde{u}_p,\tilde{f}_p\theta)\in\GCMI_{*}(S,\R^n)$, and in particular, $\Re\int_{C}\tilde{f}_p\theta=0$ for every $p\in P$.
\item[\rm (iv)] $\Im\int_{C}\tilde{f}_1\theta=\mathfrak{p}(C)$.
\item[\rm (v)] $\Re\int_{x_0}^{x} \tilde{f}_1\theta \in \Omega$ for every $x\in E$.
\end{itemize}
\end{lemma}

\begin{proof}
The proof consists of two parts. In the first one, we extend the family $f_p \in \Oscr(K,{\bf A}_*)$ to a continuous family of nonflat continuous maps $f_p\colon S\to{\bf A}_*$ ($p\in P$) meeting properties (i), (ii), (iv) and (v), whereas in the second part, we modify their real periods in order to obtain the family $\tilde{f}_p$ fulfilling conclusions of the lemma.

Recall that $f_0=f\colon M\to{\bf A}_*$ and $f_p=2\partial{u}_p/\theta\colon K\to{\bf A}_*$ ($p\in P$) are nonflat holomorphic maps.
By~\cite[Lemma~3.3]{alarcon2019interpolation}, we may extend $f_1$ to a nonflat continuous map $f_1\colon S\to{\bf A}_*$ such that $f_1$ is homotopic to $f_0$ on $S$ (note that $f_0\simeq f_1$ on $K$ by assumption), $\Re\int_{C}f_1\theta=0$, $\Im\int_{C}f_1\theta=\mathfrak{p}(C)$ and $\Re\int_{x_0}^{x}f_1\theta\in\Omega$ for every $x\in E$.
Since $f_0\simeq f_p\simeq f_1$ on $K$ ($p\in P$) and $f_0\simeq f_1$ on $S$, there exists a homotopy $f_p\colon S\to{\bf A}_*$ ($p\in P$) of nonflat continuous maps on $S$ that agree with the already given ones on $K$. However, it need not hold that $\Re\int_{C}f_p\theta=0$ for all $p\in P$.
To obtain the family $\tilde{f}_p$ ($p\in P$) with the desired properties, we follow the idea in the proof of~\cite[Claim, p.~26]{forstneric2019parametric}. 

Set $Q=\{0,1\}\subset P$.
Define a continuous family of nonflat continuous maps
\begin{equation} \label{eq:lemmaarcs-f_pt}
	f_{p}^{t}:=f_p\colon S\longrightarrow{\bf A}_*, \quad (p,t)\in P\times[0,1],
\end{equation}
i.e., $f_{p}^{t}$ is a homotopy which does not depend on $t\in[0,1]$.
Consider the closed curve $C\subset S$ as a union of closed arcs $C=C_1\cup C_2\cup C_3$, where $C_3=C\cap K$, and parametrize $C$ by a real analytic map $\gamma\colon [0,3]\to C$ such that $C_{i+1}=\gamma([i,i+1])$ for $i=0,1,2$. In particular, we have that $S=K\cup E=K\cup C$ and $E=C_1\cup C_2$. 
Furthermore, pick a number $\eta>0$ close to zero and let
\begin{equation*}
	\begin{cases}
		I_1=[\eta, 1-\eta], \\ 
		I_2=[1+\eta,2-\eta], \\
		C_{i}'=\gamma(I_i) \quad \text{for } i=1,2.
	\end{cases}
\end{equation*}
Define a continuous family of paths $\sigma_{p}\colon[0,3]\to{\bf A}_*$ ($p\in P$) by the equation
\begin{equation} \label{eq:sigma_p}
	\sigma_{p}(s)ds := f_{p} \left(\gamma(s)\right)\theta\left(\gamma(s),\gamma'(s)\right)ds, \quad s\in[0,3].
\end{equation}
(Recall that $f_{p}^{t}=f_p$~\eqref{eq:lemmaarcs-f_pt} holds on $S$.)
We shall change $\sigma_{p}$ on $I_1\cup I_2$ in such a way that the modified maps $\tilde{\sigma}_{p}^{t}$ will satisfy $\Re\int_{C}\tilde{\sigma}_{p}^{t}=0$ for all $(p,t)\in P\times[t_0,1]$, for some value $t_0\in(0,1)$. At the same time, $\sigma_{p}$ will remain unchanged on $[0,3]\setminus(I_1\cup I_2)$.

In the first step, consider the interval $I_1$. Set $\alpha_p:=\int_{C}\sigma_p \in\C^n$ ($p\in P$) and observe that $\Re(\alpha_0)=\Re(\alpha_1)=0$ and $\Im(\alpha_1)=\mathfrak{p}(C)$. Choose a continuous family $\alpha_{p}^{t}\in\C^n$, $(p,t)\in P\times[0,1]$, such that $\lim_{t\to 1}\Re(\alpha_{p}^{t})=0$ ($p\in P$) and $\alpha_{p}^{t}=\alpha_p$ for $(p,t)\in Q\times[0,1]$. Fix a number $\varepsilon>0$.
By~\cite[Lemma~3.1]{forstneric2019parametric}, there exist a homotopy of paths 
\begin{equation*}
	\sigma_{p}^{t}\colon I_1\longrightarrow{\bf A}_*, \quad (p,t)\in P\times[0,1],
\end{equation*}
such that $\sigma_{p}^{t}=\sigma_p$ near the endpoints of $I_1$ for $(p,t)\in P\times[0,1]$, $\sigma_{p}^{t}=\sigma_p$ on the whole $I_1$ for $(p,t)\in Q\times[0,1]$, and also
\begin{equation} \label{eq:lemma-arcs-sigma}
	\left|\Re\int_{0}^{3}\sigma_{p}^{t}(s)ds \right| < \varepsilon, \quad (p,t)\in P\times[t_0,1],
\end{equation}
where $t_0\in(0,1)$. We may choose $\alpha_{p}^{t}\in\C^n$ such that the inequality~\eqref{eq:lemma-arcs-sigma} holds for some value of $t_0$ (recall that $\lim_{t\to1}\Re(\alpha_{p}^{t})=0$).

In the second step, we modify $\sigma_{p}$ on $I_2$ to insure that the integrals~\eqref{eq:lemma-arcs-sigma} equal $0$.
By assumption on $f_p$, $\sigma_p$ is nonflat on $I_2$ for all $p\in P$. Therefore, \cite[Lemma~3.6]{alarcon2018every} allows us to embed $\sigma_p|_{I_2}$ as the core $\tau_{p,0}=\sigma_p$ of a period dominating spray of paths
\begin{equation*}
	\tau_{p,\zeta}\colon I_2\longrightarrow{\bf A}_*, \quad (p,\zeta)\in P\times\B,
\end{equation*}
depending holomorphically on $\zeta\in\B\subset\C^N$ for some $N\in\N$, and such that $\tau_{p,\zeta}$ is independent of parameter $\zeta$ near the endpoints of $I_2$. Provided that $\varepsilon>0$ is sufficiently small~\eqref{eq:lemma-arcs-sigma}, the implicit function theorem gives us a continuous function 
\begin{equation} \label{eq:zeta-paths}
	\zeta\colon P\times[0,1] \longrightarrow \B, \quad \zeta=0 \text{ on } Q\times[0,1],
\end{equation}
such that, defining
\begin{equation*}
\tilde{\sigma}_{p}^{t}(s) :=
	\begin{cases}
		\tau_{p,\zeta(p,t)}(s) &s\in I_2 \\
		\sigma_{p}^{t}(s) &s\in I_1 \\
		\sigma_p(s) &s\in[0,3]\setminus (I_1\cup I_2),
	\end{cases}
\end{equation*}
for every $(p,t)\in P\times[0,1]$, it satisfies that
\begin{equation*}
	\Re\int_{0}^{3}\tilde{\sigma}_{p}^{t}(s)ds = 0, \quad (p,t)\in P\times[t_0,1].
\end{equation*}
The associated continuous family of nonflat continuous maps $\tilde{f}_{p}^{t}\colon S\to{\bf A}_*$ is given by~\eqref{eq:sigma_p}, with $\sigma_p$ replaced by $\tilde{\sigma}_{p}^{t}$, $(p,t)\in P\times[0,1]$.
By properties of $\tilde{\sigma}_{p}^{t}$, it satisfies the following conditions.
\begin{itemize}
\item[\rm (i.1)] $\tilde{f}_{p}^{t}=f_{p}^{t}=f_p$ on $K$ for $(p,t)\in P\times[0,1]$.
\item[\rm (ii.1)] $\tilde{f}_{p}^{t}=f_{p}^{t}$ on $S$ for $(p,t)\in Q\times[0,1]$ (recall in particular~\eqref{eq:zeta-paths}), and hence $\tilde{f}_{0}^{t}=f$, $\tilde{f}_{1}^{t}=f_1$ on $S$ for every $t\in[0,1]$.
\item[\rm (iii.1)] $\Re\int_{C}\tilde{f}_{p}^{t}\theta=0$ for $(p,t)\in P\times[t_0,1]$.
\end{itemize}
Defining 
\begin{equation*}
	\tilde{f}_p:=\tilde{f}_{p}^{1} \quad\text{and}\quad \tilde{u}_p(x):=u_{p}(x_0) + \Re\int_{x_0}^{x} \tilde{f}_p\theta, \quad x\in S,
\end{equation*}
they meet conclusions of the lemma. Indeed, (i) and (ii) follow from (i.1) and (ii.1), respectively, whereas (iii)--(v) are insured by (iii.1), $\tilde{f}_1=f_1$ and definition of $f_1$ on $S$ in the first part of the proof.
\end{proof}


\section{Proof of Theorem~\ref{thm:main}}
\label{sec:proof-main}

Choose a smooth strongly subharmonic Morse exhaustion function $\rho\colon M\to\R$ such that every level set $\{\rho=c\}$ contains at most one critical point of $\rho$. Take an increasing sequence $c_0<c_1<c_2<\dots$ of regular values of $\rho$ with $\lim_{i\to\infty}c_i=+\infty$ and such that every set $A_i=\{c_{i-1}<\rho<c_i\}$ contains at most one critical point of $\rho$ ($i\geq 1$).
Set $M_i=\{\rho\leq c_i\}$ for $i\in\Z_+$. By general position, we may assume that $bM_0=\{\rho=c_0\}$ contains no zeros of $u$.

Set $u_{t}^{0}:=u|_{M_0}$, $t\in[0,1]$, and pick a number $\varepsilon_0>0$.
We shall inductively construct homotopies $\{u_{t}^{i}\in\CMI_{*}(M_i,\R^n)\}_{i}$, $t\in[0,1]$, of nonflat conformal minimal immersions from a neighbourhood of $M_i$ into $\R^n$, and a decreasing sequence of positive numbers $\{\varepsilon_i>0\}_{i}$, such that the following properties hold for every $i\in\N$.
\begin{itemize}
\item[\rm (1$_i$)] $||u_{t}^{i}-u_{t}^{i-1}||_{M_{i-1}}<\varepsilon_{i-1}$ for every $t\in[0,1]$.

\item[\rm (2$_i$)] $u_{0}^{i}=u|_{M_i}$.

\item[\rm (3$_i$)] $||u_{1}^{i}(x)||_\infty>i-1$ for all $x\in{M_{i}\setminus\mathring M_{i-1}}$.

\item[\rm (4$_i$)] $||u_{1}^{i}(x)||_\infty>i$ for all $x\in bM_i$.

\item[\rm (5$_i$)] $\Flux(u_{1}^{i})=\mathfrak{p}$ on $H_1(M_i,\Z)$.

\item[\rm (6$_i$)] If $n\geq5$, then $u_{1}^{i}$ is injective.

\item[\rm (7$_i$)] $0<\varepsilon_i<\varepsilon_{i-1}/2$, and every harmonic map $v\colon M_i\to\R^n$ such that $||v-u_{t}^{i}||_{M_i}<2\varepsilon_i$ for some $t\in[0,1]$ is a nonflat immersion, and is in addition injective if $n\geq5$ and the inequality holds for $t=1$.
\end{itemize}
If such sequences exist, then $\{u_{t}^{i}\}_i$, $t\in[0,1]$, converge uniformly on compact sets in $M$ to a continuous family of limit maps 
\begin{equation*}
	u_t:=\lim_{i\to\infty}u_{t}^{i} \colon M\longrightarrow\R^n, \quad t\in[0,1],
\end{equation*}
that are nonflat conformal minimal immersions for all $t\in[0,1]$, and satisfy that $u_0=u$, $u_1$ is proper, $\Flux(u_1)=\mathfrak{p}$, and $u_1$ is an embedding if $n\geq5$.
Indeed, convergence is ensured by (1$_i$) and (7$_i$), $u_0=u$ follows from (2$_i$) and $\Flux(u_1)=\mathfrak{p}$ is guaranteed by (5$_i$). Properness of $u_1$ is a consequence of (3$_i$) and (4$_i$); for calculations, we refer to the proof of~\cite[Theorem~3.10.3, p.~183]{alarcon2021minimal} with $(\tilde{x}, x^i)$ replaced by $(u_1,u_1^i)$ in our setting, and we do not repeat it here. Finally, embeddedness of $u_1$ for $n\geq5$ follows from (6$_i$) and (7$_i$).

We now explain the induction.
The base step is provided by $u_{t}^0=u|_{M_0}\in\CMI_{*}(M_0,\R^n)$, $t\in[0,1]$, and $\varepsilon_0>0$ chosen at the beginning of the proof. Observe that (2$_0$), (4$_0$) and (5$_0$) hold by assumption, whereas the rest are void.
For the inductive step, let $i\in\N$ and suppose that we have already found homotopies $u_{t}^{j}\in\CMI_{*}(M_j,\R^n)$, $t\in[0,1]$, and numbers $\varepsilon_j>0$ satisfying the required conditions for every $j\in\{0,1,\dots,i-1\}$. We now construct the next pair $(u_{t}^{i},\varepsilon_i)$.

\textit{Case~1: The domain $A_i=\{c_{i-1}<\rho<c_i\}$ does not contain any critical point of $\rho$.}
In this situation, $M_{i-1}$ is a strong deformation retract of $M_i$, hence Proposition~\ref{prop:noncritical} furnishes us with a continuous family of nonflat conformal minimal immersions $u_{t}^{i}\in\CMI_{*}(M_i,\R^n)$, $t\in[0,1]$, fulfilling (1$_i$)--(6$_i$). We choose $\varepsilon_i>0$ small enough to satisfy (7$_i$) as well.

\textit{Case~2: The domain $A_i=\{c_{i-1}<\rho<c_i\}$ contains a critical point $x$ of $\rho$.}
By assumption, $x$ is a unique critical point of $\rho$ on $A_i$ and it is a Morse point with Morse index either $0$ or $1$. We consider the following subcases.

\textit{Subcase~2.1: The Morse index of $\rho$ at $x$ equals $0$.}
Take a small smoothly bounded disc $\Delta\subset A_i$ around $x$ in $A_i$. Extend $u_{t}^{i-1}\in\CMI_{*}(M_{i-1},\R^n)$, $t\in[0,1]$, to a disjoint union $S_i=M_{i-1}\cup\Delta$ by defining any homotopy of nonflat conformal minimal immersions $u_{t}^{i-1}\in\CMI_{*}(\Delta,\R^n)$, $t\in[0,1]$, such that $u_{0}^{i-1}=u|_{\Delta}$ and $||u_{1}^{i-1}(q)||_{\infty}>i-2$ for all $q\in\Delta$. The set $S_i=M_{i-1}\cup\Delta \subset M_i$ is admissible in $M_i$ and has the same topology as $M_i$, therefore $u_{t}^{i}\in\CMI_{*}(M_i,\R^n) $ is obtained by applying Case~1.

\textit{Subcase~2.2: The Morse index of $\rho$ at $x$ equals $1$.}
The change of topology is described by adjoining to $M_{i-1}$ a smoothly embedded arc $E\subset A_i\cup bM_{i-1}$ attached transversely to $bM_{i-1}$ with the endpoints $q_0,q_1\in bM_{i-1}$ and otherwise contained in $A_i$. The set $S_i=M_{i-1}\cup E$ is admissible in $M_i$ and $H_1(S_i,\Z)=H_1(M_i,\Z)$. We distinguish the following two situations.

\textit{Subcase~2.2.1.} The arc $E$ closes inside $M_{i-1}$ to a Jordan curve $C$ such that $E=\overline{C\setminus M_{i-1}}$, i.e., the endpoints $q_0,q_1\in bM_{i-1}$ of $E$ belong to the same connected component of $M_{i-1}$. Hence, the closed curve $C\subset M_i$ appears as a new element of the homology basis of $H_1(S_i,\Z)$.
In this setting, Lemma~\ref{lemma:critical} applied to $(K, E, u_p)=(M_{i-1}, E, u_{t}^{i-1})$ furnishes us with a continuous family $u_{t}^{i-1}\in\GCMI_{*}(S_i,\R^n)$, $t\in[0,1]$, of nonflat generalized conformal minimal immersions which equals the already given family $u_{t}^{i-1}\in\CMI_{*}(M_{i-1},\R^n)$, $t\in[0,1]$, on $M_{i-1}$, and satisfies that 
\begin{equation} \label{eq:proof thm S_i}
	\begin{cases}
		u_{0}^{i-1}=u|_{S_i}, \\
		||u_{1}^{i-1}(q)||_\infty>i-2 \quad \text{for all } q\in E, \\
		\Flux(u_{1}^{i-1})(C)=\mathfrak{p}(C).
	\end{cases}
\end{equation}
Arguing as in the proof of \cite[Theorem~5.3]{alarcon2016embedded}, we may approximate $u_{t}^{i-1}\in\GCMI_{*}(S_i,\R^n)$ arbitrarily closely in the $\Cscr^1(S_i)$ topology by a continuous family $\tilde{u}_{t}^{i-1}\in\CMI_{*}(V_i,\R^n)$, $t\in[0,1]$, of nonflat conformal minimal immersions defined in an open neighbourhood $V_i\subset M_i$ of $S_i$ such that
\begin{equation*} 
	\begin{cases}
		\tilde{u}_{0}^{i-1}=u|_{V_i}, \\
		||\tilde{u}_{1}^{i-1}(q)||_\infty>i-2 \quad \text{for all } q\in V_i\setminus \mathring M_{i-1}, \\
		\Flux(\tilde{u}_{1}^{i-1})=\mathfrak{p} \quad \text{on } H_1(S_i,\Z).
	\end{cases}
\end{equation*}
(Here, take into account~\eqref{eq:proof thm S_i}.)
Choose a smaller neighbourhood $W_i\Subset V_i$ of $S_i$ such that $\overline{W}_i$ is a strong deformation retract of $V_i$ and consider $\tilde{u}_{t}^{i-1}\in\CMI_{*}(W_i,\R^n)$, $t\in[0,1]$, as the restriction to $W_i$ of the one defined above. The proof now reduces to Case~1.

\textit{Subcase~2.2.2.} The endpoints $q_0,q_1\in bM_{i-1}$ of the arc $E$ belong to two different connected components of $M_{i-1}$, that is, no new element of the homology basis of $S_i$ appears and the number of connected components decreases by one.
This subcase is dealt with similarly to Subcase~2.2.1, applying Lemma~\ref{lemma:critical}. However, property (iii) in Lemma~\ref{lemma:critical} is replaced by $\Re\int_{E}\tilde{f}_p\theta=\tilde{u}_p(q_1)-\tilde{u}_p(q_0)$, whereas property (iv) concerning the imaginary part of the integral is irrelevant.

This closes the induction and proves the theorem. 


\section{Some remarks on minimal surfaces defined on compact \\ bordered Riemann surfaces}
\label{sec:bordered}

In recent years, one of the questions in minimal surface theory was to understand which domains in $\R^3$ admit complete properly immersed minimal surfaces, and what are the conformal properties of such minimal surfaces with respect to the geometry of the domain. As a result, Alarc\'on, Drinovec Drnov\v{s}ek, Forstneri\v{c}, and L\'opez in~\cite{alarcon2019minimal} proved general existence and approximation results for complete proper conformal minimal immersions from a bordered Riemann surface into a minimally convex domain in $\R^3$.
In this section, we present how these results may be upgraded using techniques developed in the present paper.

Recall that a compact connected oriented surface $M$, endowed with a complex structure, is called a \emph{compact bordered Riemann surface} if its boundary $bM\neq\varnothing$ consists of finitely many smooth Jordan curves. The interior $\mathring M=M\setminus bM$ is said to be an (open) \emph{bordered Riemann surface}.

\begin{definition}
A domain $D\subset\R^n$ ($n\geq 3$) is \emph{minimally convex} if it admits a smooth exhaustion function $\rho\colon D\to\R$ that is \emph{minimal strongly plurisubharmonic} (or \emph{strongly $2$-plurisubharmonic}), that is, if for every $x\in D$, the sum of the two smallest eigenvalues of the Hessian $\Hess_{\rho}(x)$ is positive.
\end{definition}

It is known that every convex domain is minimally convex, however, there are examples of minimally convex domains without convex boundary points (see, e.g.~\cite[Example~1.4]{alarcon2019minimal}.)
Minimal (strongly) plurisubharmonic functions belong to a larger family of \emph{(strongly) $p$-plurisubharmonic functions}\footnote{Let $D\subset\R^n$ be a domain. A function $u\colon D\to\R$ of class $\Cscr^2(D)$ is \emph{(strongly) $p$-plurisubharmonic} for some $p\in\{1,\dots,n\}$ if the sum of the $p$ smallest eigenvalues of $\Hess_{u}(x)$ is nonnegative, or positive, respectively, for every $x\in D$.} that were studied by Harvey and Lawson~\cite{harvey2013pconvexity}.
However, we shall focus on the special case for $p=2$, in particular, on the main result~\cite[Theorem~1.1]{alarcon2019minimal} in the aforementioned paper, which asserts that every conformal minimal immersion of a compact bordered Riemann surface $M$ into a minimally convex domain $D\subset\R^3$ can be approximated, uniformly on compacts in $\mathring M=M\setminus bM$, by proper complete conformal minimal immersions $\mathring M\to D$. By combining this with Lemma~\ref{lemma:hom-2fixed}, we obtain the following corollary.

\begin{corollary} \label{cor:bordered1}
Assume that $D$ is a minimally convex domain in $\R^3$, and let $M$ be a compact bordered Riemann surface with nonempty boundary $bM$. Then, every nonflat conformal minimal immersion $u\colon M\to D$ is homotopic through nonflat conformal minimal immersions $\mathring M\to D$ to a proper complete nonflat conformal minimal immersion $\tilde{u}\colon \mathring M\to D$ with $\Flux(u)=\Flux(\tilde{u})$.
\end{corollary}

Likewise, \cite[Theorems~1.7 and~1.9]{alarcon2019minimal} may be upgraded in the same manner, and hence, providing homotopies instead of only uniform approximation on compact sets. These results hold true for any $n\geq 3$. 


\section{Directed holomorphic curves}
\label{sec:directed}

The aim of this section is to generalize Theorem~\ref{thm:main}. We prove Theorem~\ref{thm:directed}, which affirms that every nondegenerate holomorphic immersion of an open Riemann surface into $\C^n$ directed by a certain complex conical subvariety $A\subset\C^n$, meaning that the derivative of the map belongs to $A\setminus\{0\}$, can be deformed by a homotopy of maps of the same type to a nondegenerate proper directed embedding.

\begin{definition} \label{def:directed-imm}\cite[Definition~2.1]{alarcon2014null}
Let $M$ be an open Riemann surface and let $A\subset\C^n$ be a closed conical complex subvariety. We call a holomorphic immersion $F=(F_1,\dots,F_n)\colon M\to\C^n$ an \emph{$A$-immersion}, or \emph{directed by $A$}, if its complex derivative $F'$ with respect to any local holomorphic coordinate on $M$ assumes its values in $A_*=A\setminus\{0\}$.
An injective $A$-immersion is called an \emph{$A$-embedding}.
\end{definition}

Recall that a subvariety $A\subset\C^n$ is \emph{conical} if $\lambda A=A$ for every $\lambda\in\C\setminus\{0\}$, meaning that it is a union of complex lines through the origin.
In what follows, we will assume that a closed conical complex subvariety $A$ is irreducible, smooth (non-singular) away from $0\in\C^n$ and not contained in any hyperplane of $\C^n$. These conditions imply that $n\geq 3$. We denote the punctured subvariety by $A_*:=A\setminus\{0\}$, assuming that it is a smooth and connected Oka manifold; the latter being the crucial property.
Fix a nowhere vanishing holomorphic $1$-form $\theta$ on $M$~\cite{gunning1967immersion} and write $dF=f\theta$, where $f=(f_1,\dots,f_n)\colon M\to\C^n$ is a holomorphic map. Then, $F$ is directed by $A$ if and only if $f=dF/\theta$ has the range in $A_*$. 
Conversely, given a holomorphic map $f=(f_1,\dots,f_n)\colon M\to A_*$, it corresponds to a holomorphic $A$-immersion $F\colon M\to\C^n$ if and only if the vectorial $1$-form $f\theta=(f_1\theta,\dots,f_n\theta)$ is holomorphic and exact, i.e., $\int_{C}f\theta=0$ for every closed curve $C\subset M$. Furthermore, taking an arbitrary initial point $x_0\in M$, the corresponding $A$-immersion $F$ is recovered by
\begin{equation*}
F(x)=F(x_0) + \int_{x_0}^{x}f\theta, \quad x\in M.
\end{equation*}
We point out that the choice of $1$-form $\theta$ is not important (as long as it is fixed) since the subvariety $A\subset\C^n$ is conical.
A special case of directed $A$-immersions occurs when $A$ is the null quadric ${\bf A}$~\eqref{eq:null quadric}; then a holomorphic $A$-immersion is precisely the null curve.

\begin{definition} \label{def:nondegenerate}
Let $A$ be as in Definition~\ref{def:directed-imm}, $M$ be either an open Riemann surface or a compact bordered Riemann surface, and let $\theta$ be a nowhere vanishing holomorphic $1$-form on $M$.
A holomorphic $A$-immersion $F\colon M\to\C^n$ is \emph{nondegenerate} if the map $f=dF/\theta\colon M\to A_*$ is \emph{nondegenerate}, that is, if the tangent spaces $T_{f(x)}A\subset T_{f(x)}\C^n \cong \C^n$ over all points $x\in M$ span $\C^n$.
\end{definition}

We denote by $\Iscr_{A}(M,\C^n)$ the space of all holomorphic $A$-immersions $M\to\C^n$ endowed with the compact-open topology, and by
\begin{equation*}
	\Iscr_{A,*}(M,\C^n) \subset \Iscr_{A}(M,\C^n)
\end{equation*}
the subspace of $\Iscr_{A}(M,\C^n)$ consisting of all nondegenerate holomorphic $A$-immersions.
What we know is the parametric h-principle with approximation for holomorphic immersions of open Riemann surfaces directed by Oka manifolds~\cite[Theorem~5.3]{forstneric2019parametric} and the weak homotopy equivalence principle for directed immersions~\cite[Theorem~5.6]{forstneric2019parametric}. In particular, the inclusion
\begin{equation*}
\Iscr_{A,*}(M,\C^n) \longhookrightarrow \Cscr(M,A_*)
\end{equation*}
is a weak homotopy equivalence. However, an analogue of~\eqref{eq:parametric h-inclusion complete} for the space of nondegenerate complete directed immersions remains unanswered. The reason lies in the proof of~\cite[Theorem~1.1]{alarcon2025strong}, so far the only known result concerning the space $\CMI_{*}^{c}(M,\R^n)$, which strongly uses the geometry of the null quadric ${\bf A}_*\subset\C^n$. More precisely, the standard construction of complete conformal minimal immersions is based on the L\'opez-Ros deformation for minimal surfaces~\cite{lopez1991embedded} and Jorge-Xavier-type labyrinth~\cite{jorge1980complete}, which is not generalized to an arbitrary Oka cone in $\C^n$.

In the present paper, when studying the space of proper nondegenerate directed immersions $M\to\C^n$, we are able to overcome the difficulty described above to a certain extent, since we rely on the Oka theory instead of the geometry of the null quadric.
Here is our second main result.

\begin{theorem} \label{thm:directed}
Let $A\subset\C^n$ ($n\geq3$) be a closed irreducible conical complex subvariety which is not contained in any hyperplane, is smooth away from $0$, and such that $A_*=A\setminus\{0\}$ is an Oka manifold. Assume that for $k\in\{1,2\}$, the hyperplane section $A\cap\{z_k=1\}$ is an Oka manifold and the coordinate projection $\pi_k\colon A\to\C^n$ onto the $z_k$-axis admits a local holomorphic section $h_k$ near $z_k=0$ with $h_k(0)\neq0$.
Let $F\colon M\to\C^n$ be a nondegenerate holomorphic $A$-immersion from an open Riemann surface $M$ into $\C^n$. Then there exists a homotopy $F_t\colon M\to\C^n$, $t\in[0,1]$, of nondegenerate holomorphic $A$-immersions such that $F_0=F$ and $F_1$ is a proper nondegenerate holomorphic $A$-embedding.

Furthermore, we can choose the homotopy such that the first two components of $F_1$ determine a proper map $M\to\C^2$.
\end{theorem}

\begin{proof}
The proof uses the same tools as the proof of Theorem~\ref{thm:main}, therefore, we explain the main steps and point out the modifications, leaving out the details. (See Sections~\ref{sec:(non)critical},~\ref{sec:proof-main}.)

Exhaust $M$ by a sequence of connected smoothly bounded Runge compact domains
\begin{equation*}
	M_0\Subset M_1\Subset M_2 \Subset \cdots \Subset \bigcup_{j=0}^{\infty}M_j=M
\end{equation*}
that are sublevel sets of a smooth strongly subharmonic Morse exhaustion function $\rho\colon M\to\R$, meeting the properties stated in the proof of Theorem~\ref{thm:main}. Assume that the first two components of $F$ have no common zeros on $bM_0$.

Set $F_{t}^{0}:=F|_{M_0}$, $t\in[0,1]$, and pick a number $\varepsilon_0>0$. We inductively construct homotopies $\{F_{t}^{i}\in\Iscr_{A,*}(M_i,\C^n)\}_{i}$, $t\in[0,1]$, of nondegenerate $A$-immersions from a neighbourhood of $M_i$ into $\C^n$, and a decreasing sequence of positive numbers $\{\varepsilon_i>0\}_{i}$, such that the corresponding properties (analogous to (1$_i$)--(7$_i$), but omitting (5$_i$) and taking (6$_i$) for every $n\geq3$, in the cited proof) hold.
Consequently, $\{F_{t}^{i}\}_i$, $t\in[0,1]$, converge uniformly on compact sets in $M$ to a continuous family of nondegenerate $A$-immersions
\begin{equation*}
	F_t:=\lim_{i\to\infty}F_{t}^{i} \colon M\longrightarrow\C^n, \quad t\in[0,1],
\end{equation*}
that satisfy $F_0=F$ and $F_1$ is a proper nondegenerate $A$-embedding.

The noncritical case of the induction is solved by an analogous result to Proposition~\ref{prop:noncritical}, but for directed immersions. In its proof, we firstly replace \cite[Lemma~3.11.1]{alarcon2021minimal} by \cite[Lemma~8.2]{alarcon2014null}. Moreover, observe that Lemma~\ref{lemma:hom-2fixed} and Claim~\ref{claim:2fixed} hold in our setting (however, we consider the whole complex integrals instead of only the real parts in their proofs). The main reasons that enable these generalizations are that $A_*$ is an Oka manifold and that the convex hull of $A$ equals $\C^n$. (We refer to \cite[Lemma~3.1]{alarcon2014null} for the latter, and to \cite[Lemma~5.1]{alarcon2014null} for the existence of period dominating sprays of nondegenerate holomorphic maps to an Oka manifold $A_*$.)
Injectivity of $F_{1}^{i}$ ($i\in\N$) is obtained by applying~\cite[Theorem~2.5]{alarcon2014null} in place of \cite[Theorem~3.4.1]{alarcon2021minimal}.

In the critical case, we extend the continuous family of holomorphic maps 
\begin{equation*}
	f_{t}^{i-1}=dF_{t}^{i-1}/\theta\colon M_{i-1}\longrightarrow A_*, \quad t\in[0,1],
\end{equation*}
from $M_{i-1}$ either to a disjoint smoothly bounded disc $\Delta\subset\mathring M_i$ (see Subcase~2.1 in the proof of Theorem~\ref{thm:main}), or to a smoothly embedded arc $E \subset \mathring M_{i}\setminus\mathring M_{i-1}$ attached transversely to $bM_{i-1}$, by an analogue of Lemma~\ref{lemma:critical} (see Subcase~2.2.1 if $E$ closes inside $M_{i-1}$ to a loop $C$, and Subcase~2.2.2 if $E$ connects two connected components of $M_{i-1}$). As in the noncritical case, we consider the whole complex integrals instead of only the real parts. Note also that \cite[Lemma~3.3]{alarcon2019interpolation} and \cite[Lemma~3.1]{forstneric2019parametric} hold for this subvariety $A$ (the latter is justified since the convex hull of $A$ equals $\C^n$). The remaining parts of the proof resemble the critical case in the proof of Theorem~\ref{thm:main} with the null quadric ${\bf A}_*\subset\C^n$ replaced by a general Oka manifold $A_*$, using period dominating sprays of nondegenerate holomorphic maps to $A_*$ and Mergelyan approximation theorem for maps to Oka manifolds~\cite[Theorem~5.4.4]{forstneric2017stein}. See also~\cite[Claim, p.~26, and the proof of Theorem~5.3]{forstneric2019parametric}.
\end{proof}

\begin{remark}
Conditions on the cone $A\subset\C^n$, in particular those concerning the hyperplane sections $A\cap\{z_k=1\}$ for $k\in\{1,2\}$, are purely technical. They were firstly introduced in~\cite{alarcon2014null} and are necessary in order to apply \cite[Lemma~8.2]{alarcon2014null} (see the proof of Theorem~\ref{thm:directed}), which enables a directed immersion to be made proper. 
\end{remark}

In analogy to Corollary~\ref{cor:surjection CMI} and Problem~\ref{prob:whe CMI}, we state the following consequence and an open problem.

\begin{corollary} \label{cor:surjection NC}
The inclusion
\begin{equation} \label{eq:surjection NC-proper}
	\Iscr_{A,*}^{p}(M,\C^n) \longhookrightarrow \Iscr_{A,*}(M,\C^n)
\end{equation} 
of the space of proper nondegenerate holomorphic $A$-immersions $M\to\C^n$ into the space of nondegenerate holomorphic $A$-immersions $M\to\C^n$ induces a surjection of path components.
\end{corollary}

A direct consequence states that the inclusion
\begin{equation} \label{eq:surjection NC-complete}
	\Iscr_{A,*}^{c}(M,\C^n) \longhookrightarrow \Iscr_{A,*}(M,\C^n)
\end{equation} 
of the space of complete nondegenerate holomorphic $A$-immersions $M\to\C^n$ into the space of nondegenerate holomorphic $A$-immersions $M\to\C^n$ induces a surjection of path components. We point out that Corollary~\ref{cor:surjection NC} provides the first topological information on the space $\Iscr_{A,*}^{p}(M,\C^n)$, and even $\Iscr_{A,*}^{c}(M,\C^n)$.

\begin{problem} \label{prob:whe directed}
Does the inclusion~\eqref{eq:surjection NC-proper} induce a bijection of path components? Is it a weak homotopy equivalence?
What about the inclusion~\eqref{eq:surjection NC-complete}?
\end{problem}

A positive answer to Problem~\ref{prob:whe directed} would, combined with \cite[Theorem~5.6]{forstneric2019parametric}, tell us that the inclusion $\Iscr_{A,*}^{p}(M,\C^n) \longhookrightarrow \Oscr(M,A_*)$, or $\Iscr_{A,*}^{c}(M,\C^n) \longhookrightarrow \Oscr(M,A_*)$, respectively, is a weak homotopy equivalence. Equivalently, that would mean that the map $\Iscr_{A,*}^{p}(M,\C^n) \longhookrightarrow \Cscr(M,A_*)$, or $\Iscr_{A,*}^{c}(M,\C^n) \longhookrightarrow \Cscr(M,A_*)$, respectively, is a weak homotopy equivalence.

\begin{remark}
A simplification of our proof allows to extend the results to the framework of complex curves in $\C^n$ with $n\ge 2$. In particular, if $M$ is an open Riemann surface then every nondegenerate holomorphic immersion $M\to\C^n$ is homotopic through holomorphic immersions $M\to\C^n$ to a proper one, which can be chosen an embedding if $n\ge 3$.
\end{remark}


\subsection*{Acknowledgements}
This research is partially supported by the State Research Agency (AEI) via the grant no.\ PID2023-150727NB-I00, funded by MICIU/AEI/10.13039/501100011033 and ERDF/EU, Spain.





\medskip
\noindent Tja\v{s}a Vrhovnik

\smallskip
\noindent Departamento de Geometr\'{\i}a y Topolog\'{\i}a e Instituto de Matem\'aticas (IMAG), Universidad de Granada, Campus de Fuentenueva s/n, E--18071 Granada, Spain.

\smallskip
\noindent e-mail: {\tt vrhovnik@ugr.es}

\end{document}